\documentclass[11pt,leqno]{article}
\usepackage[utf8]{inputenc}
\usepackage{amsmath}
\usepackage{amsfonts}
\usepackage{amssymb}
\usepackage{graphicx}
\usepackage{mathrsfs}

\usepackage{mathrsfs}\usepackage{upref,amsthm,amsxtra,exscale}
\usepackage{cite}
\usepackage[colorlinks=true,urlcolor=blue,
citecolor=red,linkcolor=blue,linktocpage,pdfpagelabels,
bookmarksnumbered,bookmarksopen]{hyperref}
\allowdisplaybreaks[4]

\newtheorem{theorem}{Theorem}[section]

\newtheorem{remark}[theorem]{Remark}
\newtheorem{lemma}[theorem]{Lemma}
\newtheorem{proposition}[theorem]{Proposition}

\numberwithin{equation}{section}

\def\r{\mathbb{R}}
\def\rn{\mathbb{R}^N}

\def\z{\mathbb{Z}}

\def\eps{\varepsilon}
\def\rh{\rightharpoonup}
\def\io{\int_{\Omega}}
\def\irn{\int_{\r^N}}

\def\wt{\widetilde}

\def\cC{\mathcal{C}}

\def\cI{\mathcal{I}}
\def\cJ{\mathcal{J}}
\def\cM{\mathcal{M}}
\def\cN{\mathcal{N}}

\author{Mónica Clapp\footnote{M. Clapp was partially supported by UNAM-DGAPA-PAPIIT grant IN100718 and CONACYT grant A1-S-10457 (Mexico).}\; and Andrzej Szulkin}
\title{Solutions to indefinite weakly coupled cooperative elliptic systems}
\date{}

\begin{document}
\maketitle

\hfill\emph{To the memory of Andrzej Granas}

\medskip

\begin{abstract}
We study the elliptic system
\begin{equation*}
\begin{cases}
-\Delta u_1 - \kappa_1u_1 = \mu_1|u_1|^{p-2}u_1 + \lambda\alpha|u_1|^{\alpha-2}|u_2|^\beta u_1, \\
-\Delta u_2 - \kappa_2u_2 = \mu_2|u_2|^{p-2}u_2 + \lambda\beta|u_1|^\alpha|u_2|^{\beta-2}u_2, \\
u_1,u_2\in D^{1,2}_0(\Omega),
\end{cases}
\end{equation*}
where $\Omega$ is a bounded domain in $\rn$, $N\geq 3$, $\kappa_1,\kappa_2\in\r$, $\mu_1,\mu_2,\lambda>0$, $\alpha,\beta>1$, and $\alpha + \beta = p\le 2^*:=\frac{2N}{N-2}$.

For $p\in (2,2^*)$ we establish the existence of a ground state and of a prescribed number of fully nontrivial solutions to this system for $\lambda$ sufficiently large.

If $p=2^*$ and $\kappa_1,\kappa_2>0$ we establish the existence of a ground state for $\lambda$ sufficiently large if, either $N\ge5$, or $N=4$ and neither $\kappa_1$ nor $\kappa_2$ are Dirichlet eigenvalues of $-\Delta$ in $\Omega$.
\medskip

\medskip

\end{abstract}

\section{Introduction and statement of results}

We consider the elliptic system
\begin{equation} \label{syst}
\begin{cases}
-\Delta u_1 - \kappa_1u_1 = \mu_1|u_1|^{p-2}u_1 + \lambda\alpha|u_1|^{\alpha-2}|u_2|^\beta u_1, \\
-\Delta u_2 - \kappa_2u_2 = \mu_2|u_2|^{p-2}u_2 + \lambda\beta|u_1|^\alpha|u_2|^{\beta-2}u_2, \\
u_1,u_2\in D^{1,2}_0(\Omega),
\end{cases}
\end{equation}
where $\Omega$ is a bounded domain in $\rn$, $N\geq 3$, $\kappa_1, \kappa_2\in\r$, $\mu_1,\mu_2,\lambda>0$, $\alpha,\beta>1$, and $\alpha + \beta = p\le 2^*:=\frac{2N}{N-2}$. As usual, $D^{1,2}_0(\Omega)$ is the completion of $\cC^\infty_c(\Omega)$ with respect to the norm
$$\|u\|:=\left(\io |\nabla u|^2\right)^{1/2}.$$

This type of systems arise in applications (e.g., as a model for the steady states of a two-component Bose-Einstein condensate) and has attracted considerable attention in the mathematical community, beginning with the seminal paper by Lin and Wei \cite{lw}. 

The system \eqref{syst} is weakly coupled, i.e., every nontrivial solution $w_i$ to the equation 
\begin{equation} \label{eq}
-\Delta w-\kappa_iw = \mu_i|w|^{p-2}w, \quad w\in D^{1,2}_0(\Omega),
\end{equation}
$i=1,2$, gives rise to a \emph{semitrivial} solution $(w_1,0),\;(0,w_2)$ of the system. We are interested in the existence of \emph{fully nontrivial} solutions, i.e., solutions with both components $u_1$ and $u_2$ different from $0$.

If the system is positive definite, i.e., if $\kappa_1,\kappa_2<\lambda_1(\Omega)$ where $\lambda_1(\Omega)$ is the first eigenvalue of $-\Delta$ in $D^{1,2}_0(\Omega)$, it is well known that for the cubic system ($\alpha=\beta=2,\;N=3$) a positive ground state exists for sufficiently large or sufficiently small values of $\lambda>0$; see, e.g., \cite[Section 1.1]{st} and the references therein. A similar result was proved by Chen and Zou \cite{cz1,cz2} for a critical system.

On the other hand, there seem to be no results available in the literature for the indefinite case, i.e., when $\kappa_i\geq\lambda_1(\Omega)$ for some $i=1,2$. In this paper we establish, not only the existence of a ground state for any $\kappa_1, \kappa_2>0$ and $p\in (2,2^*]$, but of a prescribed number of fully nontrivial solutions when $p\in (2,2^*)$, for sufficiently large values of $\lambda$.

Note that the system \eqref{syst} is $(\z_2\times\z_2)$-invariant, where $\z_2:=\{\pm 1\}$, i.e., if $u=(u_1,u_2)$ is a solution, then every element in the $(\z_2\times\z_2)$-orbit of $u$,
$$(\z_2\times\z_2)u:=\{(u_1,u_2),\;(u_1,-u_2),\;(-u_1,u_2),\;(-u_1,-u_2)\},$$
is also a solution of \eqref{syst}. 

For $u=(u_1,u_2), v=(v_1,v_2)\in D^{1,2}_0(\Omega)\times D^{1,2}_0(\Omega)$ we set
$$B(u,v):=B_1(u_1,v_1)+B_2(u_2,v_2)$$
with 
$$B_i(u_i,v_i):=\io (\nabla u_i\cdot\nabla v_i - \kappa_i u_iv_i),\quad i=1,2.$$
Our results read as follows.

\begin{theorem} \label{mainthm}
Assume that $p\in (2,2^*)$ and $\kappa_1,\kappa_2\in\r$.
\begin{itemize}
\item[$(i)$] There exists $\Lambda_1>0$ such that for each $\lambda>\Lambda_1$ the system \eqref{syst} has a ground state solution $\bar u$ which is fully nontrivial.
\item[$(ii)$] For each positive integer $k$ there exists $\Lambda_k>0$ such that, if $\lambda> \Lambda_k$, then the system \eqref{syst} has at least $k$ $(\z_2\times\z_2)$-orbits of fully nontrivial solutions. 
\end{itemize}
Each one of these solutions $u$ satisfies 
$$0<B(u,u)<\min\{B_1(\bar w_1,\bar w_1),B_2(\bar w_2,\bar w_2)\},$$
where $\bar w_i$ is a ground state solution to equation \eqref{eq}, $i=1,2$.
\end{theorem}

This result seems to be new also in the positive definite case $\kappa_1,\kappa_2<\lambda_1(\Omega)$, and it holds true in dimensions $N=1$ and 2 as well, for $2<p<\infty$.

\begin{theorem} \label{critical}
Let $p=2^*$ and assume that $\kappa_1,\kappa_2> 0$. If $N\ge 5$, then there exists $\Lambda_1>0$ such that for each $\lambda>\Lambda_1$ the system \eqref{syst} has a ground state solution $\bar u$ which is fully nontrivial. The same conclusion remains valid if $N=4$ and $\kappa_1,\kappa_2$ are not eigenvalues of $-\Delta$ in $D^{1,2}_0(\Omega)$. Moreover,
$$0<B(\bar u,\bar u)<\min\{B_1(\bar w_1,\bar w_1),B_2(\bar w_2,\bar w_2)\},$$
where $\bar w_i$ is a ground state solution to equation \eqref{eq} with $p=2^*$, $i=1,2$.
\end{theorem}

In the positive definite case $0<\kappa_1,\kappa_2<\lambda_1(\Omega)$ Chen and Zou showed that, for $p=2^*$ and $\alpha=\beta$, the system \eqref{syst} has a ground state solution for all $\lambda>0$ if $N\geq 5$ \cite[Theorem 1.3]{cz2} and for $N=4$ if either $0<\lambda<\min\{\mu_1,\mu_2\}$ or $\lambda>\max\{\mu_1,\mu_2\}$ \cite[Theorem 1.3]{cz1}.  However, our result is new also for $0<\kappa_1,\kappa_2<\lambda_1(\Omega)$ when $\alpha\ne \beta$. Multiple positive solutions for $N=4$ were exhibited in \cite{pt} for $\kappa_1=\kappa_2=0$ and small $\lambda>0$ under suitable assumptions on the domain.

The paper is organized as follows: In Section \ref{sec:preliminaries} we introduce the variational setting. Section \ref{sec:subcritical} is devoted to the subcritical case and Section \ref{sec:critical} to the critical case. We conclude with some comments on synchronized solutions in Section \ref{sec:synchronized}. 

\section{Preliminaries} \label{sec:preliminaries}

Let $X:=D^{1,2}_0(\Omega)\times D^{1,2}_0(\Omega)$ with the usual norm
$$\|u\|:=(\|u_1\|^2+\|u_2\|^2)^{1/2},\qquad u=(u_1,u_2)\in X.$$
The solutions of the system \eqref{syst} are the critical points of the functional $\cJ_\lambda: X \to \r$ given by
\begin{align*}
\cJ_\lambda(u) :&= \frac12\io(|\nabla u_1|^2+|\nabla u_2|^2-\kappa_1|u_1|^2-\kappa_2|u_2|^2) \\
& \qquad - \frac1{p}\io(\mu_1|u_1|^{p}+\mu_2|u_2|^{p}) - \lambda \io|u_1|^{\alpha}|u_2|^{\beta} \\
&=\frac12 B(u,u) - \frac1{p}\io(\mu_1|u_1|^{p}+\mu_2|u_2|^{p}) - \lambda \io|u_1|^{\alpha}|u_2|^{\beta}.
\end{align*}
Its partial derivatives are
\begin{align*}
\partial_1\cJ_\lambda(u)v=B_1(u_1,v)-\io \mu_1|u_1|^{p-2}u_1v - \lambda\io\alpha|u_1|^{\alpha-2}|u_2|^\beta u_1v,\\
\partial_2\cJ_\lambda(u)v=B_2(u_2,v)-\io \mu_2|u_2|^{p-2}u_2v - \lambda\io\beta|u_1|^\alpha |u_2|^{\beta-2}u_2 v.
\end{align*}

It is shown in \cite[Theorem 3.1]{sw} that, for $p<2^*$, there exists a minimizer $\bar w_i$ for the energy functional 
$$J_i(w) := \frac12B_i(w,w) - \frac1{p}\io\mu_i|w|^{p}, \qquad i=1,2,$$
on the associated generalized Nehari manifold. The same is true for $p=2^*$ if $\kappa_i>0$ and, either $N\ge 5$, or $N=4$ and $\kappa_i$ is not an eigenvalue of $-\Delta$ in $D^{1,2}_0(\Omega)$ \cite[Theorem 3.6]{sww} (see also \cite{gr} and the references there). Hence, $\bar  w_i$ is a least energy nontrivial solution to the equation \eqref{eq} and $J_i(\bar w_i)>0$. Recall that such solution is called a ground state. Let
\begin{equation} \label{eq:c}
c_0 := \min\{J_1(\bar w_1), \ J_2(\bar w_2)\}.
\end{equation}

\begin{proposition} \label{prop:c}
If $\cJ_\lambda$ has a critical point $u$ such that $0<\cJ_\lambda(u)<c_0$, then $u$ is a fully nontrivial solution of \eqref{syst}.
\end{proposition}

\begin{proof}
According to \eqref{eq:c}, $\cJ_\lambda(u)\ge c_0$ for any semitrivial solution $u$ of \eqref{syst}.
\end{proof}

So our goal is to establish the existence of critical points of $\cJ_\lambda$ with critical value smaller than $c_0$. Our main abstract tool will be the following result due to Bartolo, Benci and Fortunato \cite[Theorem 2.4]{babefo}. We write it in a form which is slightly weaker and adapted for our purposes.

\begin{proposition} \label{bbf}
Let $X$ be a Hilbert space and suppose $J\in\cC^1(X,\r)$ is even and $J(0)=0$. Suppose also there exist closed subspaces $Y, Z$ of $X$ and constants $b,c_0,\rho$ such that $Y$ has finite codimension in $X$, $Z$ has finite dimension, $0<b<c_0$, $\rho>0$ and
\[
\inf\{J(u): u\in Y,\ \|u\|=\rho\} > b, \qquad \sup\{J(u): u\in Z\} < c_0.
\]
If $J$ satisfies the Palais-Smale condition at all levels $c\in(b,c_0)$ and if $k = \dim Z -\mathrm{codim}\,Y > 0$, then $J$ has either at least $k$ critical values in $(b,c_0)$ or it has infinitely many critical points in $J^{-1}(b,c_0)$.
\end{proposition}

In \cite{babefo} it is assumed that the Palais-Smale (in fact the weaker Cerami) condition is satisfied at all levels $c>0$; however, as follows from Theorems 1.3 and 2.9 there, it suffices that this holds for $c\in(b,c_0)$.

Let $0<\gamma_1<\gamma_2\leq\cdots$ be the eigenvalues of $-\Delta$ in $D^{1,2}_0(\Omega)$ counted with their multiplicity, and let $e_1,e_2,\ldots$ be the corresponding orthonormal eigenfunctions in $L^2(\Omega)$. These are also the eigenfunctions of the operator $-\Delta-\kappa_i$ in $D^{1,2}_0(\Omega)$ but the eigenvalues are shifted by $-\kappa_i$. 

For $i=1,2$, we set $X_i:=D^{1,2}_0(\Omega)$ and we write $X_i=X_i^+\oplus X_i^0\oplus X_i^-$ for the orthogonal decomposition corresponding to the positive, zero and negative part of the spectrum of $-\Delta-\kappa_i$ in $D^{1,2}_0(\Omega)$. The spaces $\wt X_i:= X_i^0\oplus X_i^-$ are finite-dimensional. Note that $X=X^+\oplus\wt X$, with
$$X^+:=X^+_1\times X^+_2\qquad\text{and}\qquad\wt X:=\wt X_1\times\wt X_2,$$
is an orthogonal decomposition of $X$ and $B$ is positive definite on $X^+$.

\section{The subcritical case} \label{sec:subcritical}

Throughout this section we assume that $p<2^*:=\frac{2N}{N-2}$.

Fix a positive integer $m$, and let $W_m$ be the subspace of $D^{1,2}_0(\Omega)$ generated by $\{e_j:1\leq j\leq m\}$, and set
$$Z_m:=\left\{(w,w):w\in W_m\right\}.$$

\begin{lemma} \label{YZ}
Let $c_0$ be as in \eqref{eq:c}. 
\begin{itemize}
\item[$(i)$] For each fixed $\lambda>0$ there exist $b,\rho>0$, $b<c_0$, such that $\inf\{\cJ_\lambda(u): u\in X^+ \text{ and } \|u\|=\rho\} > b$.
\item[$(ii)$] For each positive integer $m$ there exists $\bar \Lambda_m$ such that, if $\lambda>\bar\Lambda_m$, then $\max_{Z_m}\cJ_\lambda < c_0$. 
\end{itemize}
\end{lemma}

\begin{proof}
$(i)$ For each $\lambda>0$ we have $\cJ_\lambda(u) = \frac12 B(u,u)+o(\|u\|^2)$ as $u\in X^+$, $u\to 0$.  Note that $B(\cdot,\cdot)^{1/2}$ is a norm in $X^+$ which is equivalent to $\|\cdot\|$. So we can find $b,\rho>0$ as required.

$(ii)$ For $u=(w,w)\in Z_m$,
\begin{align*}
\cJ_\lambda(u)&=\io \left(|\nabla w|^2-\frac{\kappa_1+\kappa_2}{2}w^2\right)-\frac{1}{p}\io(\mu_1+\mu_2)|w|^p-\lambda\io |w|^p \\
&\leq \io \left(|\nabla w|^2-\frac{\kappa_1+\kappa_2}{2}w^2\right)-\frac{1}{p}\io(\mu_1+\mu_2)|w|^p=:J(w).
\end{align*}

If $\gamma_m\leq\frac{\kappa_1+\kappa_2}{2}$, then $\cJ_\lambda(u)\leq 0$ for all $u\in Z_m$ and $(ii)$ is true for $\bar\Lambda_m=0$. 

Let us now assume that $\gamma_m>\frac{\kappa_1+\kappa_2}{2}$. 
Since $W_m$ is finite-dimensional, there exist $R>0$, independent of $\lambda$, such that 
\begin{align*}
\cJ_\lambda(u)\leq J(w)\leq 0\qquad\text{if }w\in W_m\text{\; and \;}\|w\|\geq R,
\end{align*}
and $u_\lambda=(w_\lambda,w_\lambda)$ with $w_\lambda\in W_m$ and $\|w_\lambda\|<R$ such that
$$0<\max_{Z_m}\cJ_\lambda=\cJ_\lambda(u_\lambda).$$
Therefore, 
\begin{align*}
0&=\cJ'_\lambda(u_\lambda)\,u_\lambda\\
&=2\io\left(|\nabla w_\lambda|^2-\frac{\kappa_1+\kappa_2}{2}w_\lambda^2\right) - \io(\mu_1+\mu_2)|w_\lambda|^p-\lambda p\io |w_\lambda|^p.
\end{align*}
It follows that
$$\io |w_\lambda|^p\leq\frac{2}{\lambda p}\io\left(|\nabla w_\lambda|^2-\frac{\kappa_1+\kappa_2}{2}w_\lambda^2\right)\leq\frac{C}{\lambda}$$
for some positive constant $C$, independent of $\lambda$. Hence, $w_\lambda\to 0$ in $L^p(\Omega)$ as $\lambda\to\infty$ and, since $\dim W_m<\infty$, we conclude that $\|w_\lambda\|\to 0$ as $\lambda\to\infty$. This implies that
$$\cJ_\lambda(u_\lambda)\to 0\qquad\text{as \ }\lambda\to\infty,$$
which immediately yields $(ii)$.
\end{proof}

\begin{lemma} \label{ps}
$\cJ_\lambda$ satisfies the Palais-Smale condition in $X$.
\end{lemma}

\begin{proof}
This is a variant of a well known argument but for the reader's convenience we include it. Let $(u_n)$ be a Palais-Smale sequence for $\cJ_\lambda$ and let $q\in(2,p)$. Then, there exists a constant $C$ such that for almost all $n$
\begin{align*}
C + \|u_n\|  &\ge \cJ_\lambda(u_n) -\frac1q\cJ'_\lambda(u_n)u_n \\ 
&\ge \left(\frac12-\frac1q\right)B(u_n,u_n) + \left(\frac1q-\frac1p\right)\io(\mu_1|u_{n,1}|^{p}+\mu_2|u_{n,2}|^{p}).
\end{align*}
Since $B(\cdot,\cdot)^{1/2}$ and $\|\cdot\|$ are equivalent norms in $X^+$, $\wt X$ is finite-dimensional and $p>2$, we have that $(u_n)$ is bounded in $X$. So, passing to a subsequence, $u_n\rh u$ weakly in $X$,  $u_n\to u$ strongly in $L^2(\Omega)$ and $L^p(\Omega)$, and $\wt u_n\to\wt u$ strongly in $\wt X$, where $u_n=u_n^++\wt u_n$ and $u=u^++\wt u$ with $u_n^+,u^+\in X^+$ and $\wt u_n,\wt u\in\wt X$. Also, it is easy to see that $u$ is a critical point of $\cJ_\lambda$. By the Hölder inequality and the boundedness of $(u_n)$ we have
\begin{align*}
\left|\io(|u_{n,1}|^{p-2}u_{n,1}-|u_1|^{p-2}u_1)(u_{n,1}-u_1)\right| &\le C_1|u_{n,1}-u_1|_p\to 0,\\
 \left|\io(|u_{n,1}|^{\alpha-2}|u_{n,2}|^{\beta}u_{n,1}-|u_1|^{\alpha-2}|u_2|^{\beta}u_1)(u_{n,1}-u_1)\right| &\le C_2|u_{n,1}-u_1|_p\to 0, 
\end{align*}
where $|\cdot|_p$ is the norm in $L^p(\Omega)$. A similar conclusion follows with the roles of $u_1$ and $u_2$ interchanged. Hence,
\begin{align*}
o(1) &= (\cJ'_\lambda(u_n)-\cJ'_\lambda(u))[u_n-u] \\
&= B(u_n-u,\,u_n-u) + o(1)=B(u_n^+-u^+,\,u_n^+-u^+) + o(1).
\end{align*}
It follows that $u_n^+\to u^+$ in $X$. Thus, $u_n\to u$ in $X$, as claimed.
\end{proof}

Let 
\[
\cN_\lambda := \{u\in X\smallsetminus \wt X: \cJ'_\lambda(u)[tu+v]=0 \text{ for all } t\in\r,\ v\in \wt X\}.
\]
This set has been introduced by Pankov in \cite{pa} (see also \cite{sw}) and it is called the generalized Nehari manifold. However, we do not know whether it is a manifold under our present assumptions. Note that all nontrivial solutions to \eqref{syst} are necessarily contained in $\cN_\lambda$. Note also that, if $u\in\cN_\lambda$, then
\begin{equation} \label{>0}
\cJ_\lambda(u) = \cJ_\lambda(u)-\frac12\cJ'_\lambda(u)u \ge \left(\frac12-\frac1p\right)\io(\mu_1|u_1|^{p}+\mu_2|u_2|^{p}) > 0.
\end{equation} 

\begin{lemma} \label{nehari}
The set $\cN_\lambda$ is closed in $X$ and bounded away from $\wt X$. 
\end{lemma}

\begin{proof}
Arguing by contradiction, assume there exists a sequence $(u_n)$ in $\cN_\lambda$ such that $u_n^+\to 0$. Write $u_n=u_n^++u_n^0+u_n^-$ with $u_n^\pm\in X_1^\pm\times X_2^\pm=:X^\pm$ and $u_n^0\in X_1^0\times X_2^0=:X^0$. Since $B(u_n^+,u_n^+)\to 0$ and $B(u_n^-,u_n^-)\le 0$, the inequality
\[
0 = \cJ'_\lambda(u_n)u_n \le B(u_n^+,u_n^+)+B(u_n^-,u_n^-) - \io(\mu_1|u_{n,1}|^p+\mu_2|u_{n,2}|^p)
\]
implies that $u_n^-\to 0$ and $u_n^0\to 0$. So $u_n\to 0$.
As $\cJ'_\lambda(u_n)u_n = \cJ'_\lambda(u_n)u_n^0 = \cJ'_\lambda(u_n)u_n^- =0$, also $\cJ'_\lambda(u_n)u_n^+=0$ and therefore, using the Hölder and the Sobolev inequalities, we have
\begin{align*}
B_1(u_{n,1}^+,u_{n,1}^+) &= \io \mu_1|u_{n,1}|^{p-2}u_{n,1}u_{n,1}^+ + \lambda \alpha\io|u_{n,1}|^{\alpha-2}|u_{n,2}|^{\beta}u_{n,1}u_{n,1}^+ \\
&\le C_1\|u_n\|^p.
\end{align*}
Hence, $\|u_{n,1}^+\|^2 \le C\|u_n\|^p$ and, similarly, $\|u_{n,2}^+\|^2, \|u_{n,1}^-\|^2, \|u_{n,2}^-\|^2 \le C\|u_n\|^p$ for some $C>0$. Therefore,
\begin{equation} \label{ineq}
\|u_n^+\|^2+\|u_n^-\|^2 \le \overline C \|u_n\|^p.
\end{equation}
If $X^0=\{0\}$, then inequality \eqref{ineq} implies $\|u_n\|^{p-2}\ge \overline C^{-1}$. This is a contradiction. If $X^0\neq\{0\}$, we set $v_n := \frac{u_n}{\|u_n\|}$. By \eqref{ineq}, $v_n^\pm\to 0$, so $v_n^0\to v^0\ne 0$ and we may assume $v_1^0\ne 0$. Since $\cJ'_\lambda(u_n)u_n^0=0$, dividing $\partial_1\cJ_\lambda(u_n)\,u_{n,1}^0$ by $\|u_n\|^p$ we obtain
\[
0 = \io \mu_1|v_{n,1}|^{p-2}v_{n,1}v_{n,1}^0 + \lambda\alpha\io|v_{n,1}|^{\alpha-2}|v_{n,2}|^{\beta}v_{n,1}v_{n,1}^0.
\]
So passing to the limit as $n\to\infty$, we get
\[
0 = \io \mu_1|v_1^0|^p + \lambda\alpha\io|v_1^0|^{\alpha}|v_2^0|^{\beta}\geq \io \mu_1|v_1^0|^p,
\]
a contradiction. We have shown that $\cN_\lambda$ is bounded away from $\wt X$ and it follows immediately that $\cN_\lambda$ is closed. 
\end{proof}

\begin{proof} [Proof of Theorem \ref{mainthm}]
We show $(ii)$ first. Given $k\geq 1$, set $m:=k+\mathrm{codim}\,X^+$ and define $\Lambda_k:=\bar\Lambda_m$ as in Lemma \ref{YZ}$(ii)$. Set $Y:=X^+$ and $Z:=Z_m$. Let $\lambda>\Lambda_k$, and for this $\lambda$, choose $b,\rho>0$, $b<c_0$, as in Lemma \ref{YZ}$(i)$. Since $\cJ_\lambda$ is $(\z_2\times\z_2)$-invariant and satisfies the Palais-Smale condition, it follows from Proposition \ref{bbf} that the system \eqref{syst} has at least $k$ $(\z_2\times\z_2)$-orbits of nontrivial solutions $u_j$ such that $\cJ_\lambda(u_j)\in (b,c_0)$. According to Proposition \ref{prop:c}, $u_j$ are fully nontrivial.

To show $(i)$, let $\lambda>\Lambda_1$ and let $(u_n)$ be a sequence of nontrivial solutions to \eqref{syst} such that $\cJ_\lambda(u_n)\to \inf\{\cJ_\lambda(u): u\ne 0 \text{ and } \cJ'_\lambda(u)=0\}<c_0$. This is a Palais-Smale sequence for $\cJ_\lambda$. Hence, after passing to a subsequence,  $u_n\to \bar u$ in $X$. According to Lemma \ref{nehari}, $\bar u\in\cN_\lambda$. Hence, $\bar u$ is a nontrivial least energy solution to the system \eqref{syst} and $\cJ_\lambda(\bar u)>0$, as shown in \eqref{>0}. Since $\cJ_\lambda(\bar u)<c_0$, $\bar u$ is fully nontrivial. 

The last statement of the theorem follows because for each solution $u$ obtained above we have
\[
c_0 > \cJ_\lambda(u) = \cJ_\lambda(u) - \frac1p \cJ_\lambda'(u)u = \left(\frac12-\frac1p\right)B(u,u),
\]
while $\bar w_i$ satisfy
\[
c_0\le J_i(\bar w_i) =  J_i(\bar w_i) - \frac1p J_i'(\bar w_i)\bar w_i = \left(\frac12-\frac1p\right)B_i(\bar w_i,\bar w_i).
\] 
This finishes the proof.
\end{proof}

\section{The critical case} \label{sec:critical}

We begin by studying the limit system in $\rn$,
\begin{equation}
\begin{cases} \label{syst_RN*}
-\Delta u_1 = \mu_1|u_1|^{2^*-2}u_1 + \lambda\alpha|u_1|^{\alpha-2}|u_2|^\beta u_1, \\
-\Delta u_2 = \mu_2|u_2|^{2^*-2}u_2 + \lambda\beta|u_1|^\alpha|u_2|^{\beta-2}u_2, \\
u_1,u_2\in D^{1,2}(\rn).
\end{cases}
\end{equation}
We write 
\begin{align*}
\cI_{\infty,\lambda}(u) :&= \frac{1}{2}(\|u_1\|^2+\|u_2\|^2)-\frac1{2^*}(\mu_1|u_1|_{2^*}^{2^*}+\mu_2|u_2|_{2^*}^{2^*}) - \lambda \io|u_1|^{\alpha}|u_2|^{\beta}, \\
\cM_{\infty,\lambda} :&= \{u\in D^{1,2}(\rn)\times D^{1,2}(\rn):u\ne (0,0),\;\cJ'_{\infty,\lambda}(u)u=0\},
\end{align*}
for the functional and the Nehari manifold associated to \eqref{syst_RN*}, where $\|\cdot\|$ and $|\cdot|_{2^*}$ are the usual norms in $D^{1,2}(\rn)$ and $L^{2^*}(\rn)$. Note that $\cM_{\infty,\lambda}$ contains all nontrivial solutions to \eqref{syst_RN*}, also the semitrivial ones. Then,
$$\inf_{u\in\cM_{\infty,\lambda}}\cI_{\infty,\lambda}(u)=\frac{1}{N}S_{\infty,\lambda}^{N/2},$$
where
$$S_{\infty,\lambda}:=\inf_{\substack{u_1,u_2\in D^{1,2}(\rn)\\ (u_1,u_2)\neq (0,0)}}\frac{\|u_1\|^2 + \|u_2\|^2}{\left(\int_{\mathbb{R}^N}\left(\mu_1|u_1|^{2^*}+\mu_2|u_2|^{2^*} + 2^*\lambda |u_1|^{\alpha}|u_2|^{\beta}\right)\right)^{2/2^*}}.$$
To estimate $S_{\infty,\lambda}$ from above we take $u_1:=sU_\varepsilon$ and $u_2:=tU_\varepsilon$ with $s,t\geq 0$, where
\[
U_\eps(x) := [N(N-2)]^\frac{N-2}{4}\left(\frac{\eps}{\eps^2+|x|^2}\right)^\frac{N-2}{2}.
\] 
As $\|U_\eps\|^2=S^{N/2}=|U_\eps|_{2^*}^{2^*}$ for any $\varepsilon>0$, we have that
\begin{align}
S_{\infty,\lambda} &\le \inf_{\substack{s,t\geq 0\\(s,t)\ne (0,0)}}\frac{s^2 + t^2}{(\mu_1 s^{2^*}+\mu_2 t^{2^*} + 2^*\lambda s^{\alpha}t^{\beta})^{2/2^*}}S \label{eq:S_infty0} \\ 
&\leq \frac{2}{(\mu_1+\mu_2+2^*\lambda)^{2/2^*}}S, \nonumber
\end{align}
where $S$ is the best constant for the embedding $D^{1,2}(\rn)\hookrightarrow L^{2^*}(\rn)$. Using the second inequality in \eqref{eq:S_infty0} we see that there exists $\Lambda_0>0$ such that, for any $\lambda>\Lambda_0$,
\begin{align*}
\min\left\{\mu_1^{-2/2^*},\mu_2^{-2/2^*}\right\}S &>\inf_{\substack{s,t\geq 0\\(s,t)\ne (0,0)}}\frac{s^2 + t^2}{(\mu_1 s^{2^*}+\mu_2 t^{2^*} + 2^*\lambda s^{\alpha}t^{\beta})^{2/2^*}}S \\
&=\inf_{s,t>0}\frac{s^2 + t^2}{(\mu_1 s^{2^*}+\mu_2 t^{2^*} + 2^*\lambda s^{\alpha}t^{\beta})^{2/2^*}}S\geq S_{\infty,\lambda}. 
\end{align*}

\begin{proposition} \label{prop:lambda}
For any $\lambda>\Lambda_0$, one has 
\begin{align} 
0<S_{\infty,\lambda}=\inf_{s,t>0}\frac{s^2 + t^2}{(\mu_1 s^{2^*}+\mu_2 t^{2^*} + 2^*\lambda s^{\alpha}t^{\beta})^{2/2^*}}S, \label{eq:S_infty}
\end{align}
and there exist $s_\lambda,t_\lambda>0$ such that $(s_\lambda  U_\varepsilon,t_\lambda U_\varepsilon)$ solves \eqref{syst_RN*} and
\begin{equation} \label{stlambda}
\cI_{\infty,\lambda}(s_\lambda  U_\varepsilon,t_\lambda U_\varepsilon)=\inf_{u\in\cM_{\infty,\lambda}}\cI_{\infty,\lambda}(u)=\frac{1}{N}S_{\infty,\lambda}^{N/2}
\end{equation}
for every $\varepsilon>0$. 
\end{proposition}

\begin{proof}
The inequality $S_{\infty,\lambda}>0$ was proved in \cite[Lemma 3.3]{cf}. To prove the equality in \eqref{eq:S_infty} we take $u_n=(u_{n,1},u_{n,2})\in\cM_{\infty,\lambda}$ such that 
\begin{align*}
N\cI_{\infty,\lambda}(u_n)&=\|u_{n,1}\|^2+\|u_{n,2}\|^2\\
& = \mu_1|u_{n,1}|_{2^*}^{2^*}+\mu_2|u_{n,2}|_{2^*}^{2^*}+2^*\lambda|u_{n,1}|_{2^*}^\alpha |u_{n,2}|_{2^*}^\beta \to S_{\infty,\lambda}^{N/2}.
\end{align*}
Since
\begin{align*}
S(|u_{n,1}|_{2^*}^2+|u_{n,2}|_{2^*}^2) &\le \|u_{n,1}\|^2+\|u_{n,2}\|^2 \\
&=\mu_1|u_{n,1}|_{2^*}^{2^*}+\mu_2|u_{n,2}|_{2^*}^{2^*}+2^*\lambda\int_{\rn} |u_{n,1}|^\alpha |u_{n,2}|^\beta \\
&\leq \mu_1|u_{n,1}|_{2^*}^{2^*}+\mu_2|u_{n,2}|_{2^*}^{2^*}+2^*\lambda|u_{n,1}|_{2^*}^\alpha |u_{n,2}|_{2^*}^\beta,
\end{align*}
we have that
\begin{align*}
\frac{|u_{n,1}|_{2^*}^2+|u_{n,2}|_{2^*}^2}{(\mu_1|u_{n,1}|_{2^*}^{2^*}+\mu_2|u_{n,2}|_{2^*}^{2^*}+2^*\lambda|u_{n,1}|_{2^*}^\alpha |u_{n,2}|_{2^*}^\beta)^{2/2^*}}S &\le (N\cI_{\infty,\lambda}(u_n))^{2/N} \\
&= S_{\infty,\lambda}+o(1).
\end{align*}
Taking $s=|u_{n,1}|_{2^*}$, $t=|u_{n,2}|_{2^*}$  we see that this inequality, together with \eqref{eq:S_infty0}, yields the equality in \eqref{eq:S_infty}.

Setting $r=\frac{s}{t}$ we obtain
$$S_{\infty,\lambda}=\inf_{r>0}\frac{r^2 + 1}{(\mu_1 r^{2^*}+\mu_2 + 2^*\lambda r^{\alpha})^{2/2^*}}S.$$
The function
$$f_\lambda(r):=\frac{r^2 + 1}{(\mu_1 r^{2^*}+\mu_2 + 2^*\lambda r^{\alpha})^{2/2^*}}$$
satisfies $f_\lambda(0)=\mu_2^{-2/2^*}$ and\;   $\lim_{r\to\infty}f_\lambda(r)=\mu_1^{-2/2^*}$. Therefore, as 
$$\inf_{r>0}f_\lambda(r)<\min\{\mu_1^{-2/2^*},\mu_2^{-2/2^*}\}\qquad\text{for all }\lambda>\Lambda_0,$$
there exists $r_\lambda\in (0,\infty)$ such that $f_\lambda(r_\lambda)=\inf_{r>0}f_\lambda(r)$. Let $t_\lambda>0$ be such that $t_\lambda(r_\lambda U_\eps, U_\eps)\in \cM_{\infty,\lambda}$ and set $s_\lambda:=r_\lambda t_\lambda$. Then
$$\cI_{\infty,\lambda}(s_\lambda  U_\varepsilon,t_\lambda U_\varepsilon))=\frac{1}{N}S_{\infty,\lambda}^{N/2}=\inf_{u\in\cM_{\infty,\lambda}}\cI_{\infty,\lambda}(u),$$
as claimed.
\end{proof}

\begin{remark}
\emph{If $\alpha=\beta=2^*/2$, Chen and Zou showed that $\inf_{\cM_{\infty,\lambda}}\cI_{\infty,\lambda}$ is attained for all $N\geq 5$ \cite[Theorem 1.6]{cz2} and for $N=4$ if either $0<\lambda<\min\{\mu_1,\mu_2\}$ or $\lambda>\max\{\mu_1,\mu_2\}$ \cite[Theorem 1.5]{cz1}. Results for more general $\alpha,\beta$ and $\lambda=2^*$ may be found in \cite{ppw}.}
\end{remark}

Next we turn our attention to the critical system
\begin{equation} \label{syst*}
\begin{cases}
-\Delta u_1 - \kappa_1u_1 = \mu_1|u_1|^{2^*-2}u_1 + \lambda\alpha|u_1|^{\alpha-2}|u_2|^\beta u_1, \\
-\Delta u_2 - \kappa_2u_2 = \mu_2|u_2|^{2^*-2}u_2 + \lambda\beta|u_1|^\alpha|u_2|^{\beta-2}u_2, \\
u_1,u_2\in D^{1,2}_0(\Omega),
\end{cases}
\end{equation}
in a bounded domain $\Omega$. Recall the notation introduced in Section \ref{sec:preliminaries}, where now $p=2^*$.

By \eqref{eq:S_infty0} we may choose $\Lambda_1\geq\Lambda_0$ such that $\frac1N S_{\infty,\lambda}^{N/2}<c_0$ for $\lambda>\Lambda_1$. Recall  from Proposition \ref{prop:c} that, if $u$ is a critical point of $\cJ_\lambda$ and $0<\cJ_\lambda(u)<c_0$, then $u$ is fully nontrivial. 

\begin{lemma} \label{below}
$\cJ_\lambda$ satisfies the Palais-Smale condition below the level $\frac1N S_{\infty,\lambda}^{N/2}$.
\end{lemma}

\begin{proof}
Let $(u_n)$ be a Palais-Smale sequence for $\cJ_\lambda$ with $\cJ_\lambda(u_n)\to c < \frac1NS_{\infty,\lambda}^{N/2}$. As in the proof of Lemma \ref{ps} we see that $(u_n)$ is bounded, so we may assume $u_n\rh u_0$ weakly in $X$, $u_n\to u_0$ strongly in $L^2(\Omega,\r^2)$ and a.e. in $\Omega$. It is easy to show that $u_0$ is a critical point of $\cJ_\lambda$. Set $v_n:=u_n-u_0$. Using \cite[Lemmas A.2 and A.4]{cf} we get that 
$$c+o(1)=\cJ_\lambda(u_n) = \cI_\lambda(v_n) + \cJ_\lambda(u_0)+o(1)$$
where 
\begin{align*}
\cI_\lambda(v_n)&:=\frac12 (\|v_{n,1}\|^2+\|v_{n,1}\|^2) - \frac1{2^*}\io(\mu_1|v_{n,1}|^{2^*}+\mu_2|v_{n,2}|^{2^*})\\
&\qquad- \lambda \io|v_{n,1}|^{\alpha}|v_{n,2}|^{\beta},
\end{align*}
and 
$$o(1)=\cJ_\lambda'(u_n) = \cI_\lambda'(v_n) + \cJ_\lambda'(u_0)+o(1) = \cI_\lambda'(v_n)+o(1).$$
Hence, setting $a:=N(c-\cJ_\lambda(u_0))$ and using $\cI_\lambda'(v_n)v_n=o(1)$, we obtain
\[
\|v_{n,1}\|^2+\|v_{n,1}\|^2\to a,\qquad\io(\mu_1|v_{n,1}|^{2^*}+\mu_2|v_{n,2}|^{2^*}+2^*\lambda|v_{n,1}|^{\alpha}|v_{n,2}|^{\beta})\to a.
\]
If $a\ne 0$ then, by the definition of $S_{\infty,\lambda}$, 
\begin{align*}
S_{\infty,\lambda} &\le\frac{\|v_{n,1}\|^2 + \|v_{n,2}\|^2}{\left(\io(\mu_1|v_{n,1}|^{2^*}+\mu_2|v_{n,2}|^{2^*}+2^*\lambda|v_{n,1}|^{\alpha}|v_{n,2}|^{\beta})\right)^{2/2^*}}\\
&=a^{2/N}+o(1)\le (Nc)^{2/N}+o(1)<S_{\infty,\lambda}.
\end{align*}
This is a contradiction. Therefore, $a=0$, i.e., $u_n\to u_0$ strongly in $X$, as claimed.
\end{proof}

\begin{lemma} \label{lem:sww}
Let $\omega$ be an open nonempty subset of $\Omega$. 
\begin{itemize}
\item[$(i)$] If $(w_1,w_2)\in \wt X$ and $w_i=0$ a.e. in $\omega$, then $w_i=0$ a.e. in $\Omega$. Consequently, $|\cdot|_{L^{2^*}(\omega)}$ is a norm in $\wt X$, and it is equivalent to any other norm because $\dim \wt X<\infty$.
\item[$(ii)$] There exists $C>0$ such that
$$\int_\omega |w_1|^\alpha|w_2|^\beta\geq C \|w_1\|^\alpha\|w_2\|^\beta\qquad \forall (w_1,w_2)\in \wt X.$$
\end{itemize}
\end{lemma}

\begin{proof}
$(i)$ is proved in \cite[Lemma 3.3]{sww}.

$(ii):$ Arguing by contradiction, assume there exist $(w_{1,n},w_{2,n})\in \wt X$ such that $\|w_{1,n}\|=\|w_{2,n}\|=1$ and $\int_\omega|w_{1,n}|^\alpha|w_{2,n}|^{\beta}\to 0$. Since $\wt X$ is finite-dimensional, passing to a subsequence, we have that $w_{i,n}\to w_i$ for $i=1,2$. Then, $\|w_1\|=\|w_2\|=1$ and $\int_\omega|w_1|^\alpha|w_2|^{\beta}=0$, which is impossible by $(i)$. This proves the claim.
\end{proof}

Without loss of generality, we assume that $0\in\Omega$. We fix a radial cut-off function $\psi\in \cC_c^\infty(\Omega)$ such that $\psi=1$ for $|x|\le\delta$, $\delta>0$ sufficiently small. The following estimates are well known; see, e.g., \cite{bn,w}.

\begin{lemma} \label{lem:bn}
Set $\bar{u}_\eps:=\psi U_\eps$, $\eps>0$. Then, as $\eps\to 0$,
\begin{align*}
&\io|\nabla \bar{u}_\eps|^2=\irn|\nabla U_\eps|^2+O(\eps^{N-2}),\\
&\io\bar{u}_\eps^{2^*}=\irn U_\eps^{2^*}+O(\eps^{N}),\\
&\io\bar{u}_\eps^{2^*-1}=O(\eps^{\frac{N-2}{2}}),\quad \io\bar{u}_\eps=O(\eps^{\frac{N-2}{2}}),\quad \io|\nabla\bar{u}_\eps|=O(\eps^{\frac{N-2}{2}}),\\
&\io\bar{u}_\eps^{2^*-2}=O(\eps^2)\text{\; if \;}N\geq 5,\quad \io\bar{u}_\eps^{2^*-2}=O(\eps^2|\ln\eps|) \text{\; if \;} N=4,\\
&\io\bar{u}_\eps^2\geq 
\begin{cases}
d\eps^2|\ln\eps|+O(\eps^2) &\text{if }N=4,\\
d\eps^2+O(\eps^{N-2}) &\text{if }N\geq 5,
\end{cases}
\end{align*}
where $d$ is a positive constant that depends on $N$.
\end{lemma}

\begin{proof}[Proof of Theorem \ref{critical}]
Fix $\lambda>\Lambda_1$. For $\eps>0$, let $u_{\eps,1}:=s_\lambda\bar{u}_\eps$ and $u_{\eps,2}:=t_\lambda\bar{u}_\eps$  with $\bar{u}_\eps=\psi U_\eps$ as above and $s_\lambda,t_\lambda>0$ as in \eqref{stlambda}. Set $u_\eps=(u_{\eps,1},u_{\eps,2})$. We apply Proposition \ref{bbf} with $Y = X^+$ and
$$Z:=\{tu_\eps + w: t\in\r,\; w\in\wt X\}.$$
Next we show that, for $\eps$ small enough,
\begin{equation}\label{eq:claim}
\sup_Z \cJ_\lambda < \frac1NS_{\infty,\lambda}^{N/2}.
\end{equation}
We follow the proof of \cite[Lemma 3.5]{sww}, but the argument now is more delicate due the presence of the interaction term. 

Since $\kappa_1,\kappa_2> 0$, Proposition \ref{prop:lambda} and Lemma \ref{lem:bn} yield the existence of a constant $C>0$ such that
\begin{align} 
\max_{t>0}\cJ_\lambda(tu_\eps) & = \frac1N\left( \frac{(s_\lambda^2+t_\lambda^2)\io|\nabla\bar u_\eps|^2 - (\kappa_1s_\lambda^2+\kappa_2t_\lambda^2)\io\bar u_\eps^2}{(\mu_1s_\lambda^{2^*}+\mu_2t_\lambda^{2^*}+2^*\lambda s_\lambda^\alpha t_\lambda^\beta)^{2/2^*}(\io\bar u_\eps^{2^*})^{2/2^*}} \right)^{N/2} \label{eq:1} \\
& \le 
\begin{cases}
\frac1N\left(S_{\infty,\lambda} - C\eps^2+o(\eps^2)\right)^{N/2} &\text{if }N\geq 5,\nonumber\\
\frac1N\left(S_{\infty,\lambda} - C\eps^2|\ln\eps|+O(\eps^2)\right)^2 &\text{if }N=4.\nonumber
\end{cases} \\
& < \frac1N S_{\infty,\lambda}^{N/2} \nonumber
\end{align}
for every $\eps>0$ small enough.

Let now $tu_\eps+w\in Z$ with $w=(w_1,w_2)\in \wt X\smallsetminus\{(0,0)\}$ and set $\omega := \Omega\smallsetminus\text{supp\,}\psi$. We may assume $t>0$. The computations below become simpler if $w_1=0$ or $w_2=0$, as several terms will vanish.

From the convexity of the function $x\mapsto|a+x|^q$ one easily sees that
$$|a+b|^q\geq a^q+qa^{q-1}b\qquad\forall b\in\r,\;a\geq 0,\;q\geq 1.$$
Therefore, using Lemma \ref{lem:sww} we obtain
\begin{align} 
&B_i(tu_{\eps,i}+w_i,tu_{\eps,i}+w_i) \label{bilinear} \\ 
&\qquad = B_i(tu_{\eps,i},tu_{\eps,i}) +2B_i(tu_{\eps,i},w_i) +  B_i(w_i,w_i),\quad i=1,2, \nonumber \\
&\io|tu_{\eps,i}+w_i|^{2^*} = \int_{\Omega\smallsetminus\omega}|tu_{\eps,i}+w_i|^{2^*} + \int_{\omega}|w_i|^{2^*} \label{convex} \\
&\qquad \ge \io |tu_{\eps,i}|^{2^*} + \io 2^*t^{2^*-1}u_{\eps,i}^{2^*-1}w_i+c_1\|w_i\|^{2^*},\quad i=1,2, \nonumber \\
& \io|tu_{\eps,1}+w_1|^{\alpha}|tu_{\eps,2}+w_2|^{\beta} \label{mixed} \\
&\qquad = \int_{\Omega\setminus\omega}|tu_{\eps,1}+w_1|^{\alpha}|tu_{\eps,2}+w_2|^{\beta} + \int_\omega|w_1|^\alpha|w_2|^{\beta}  \nonumber\\
&\qquad\ge \io(t^{\alpha}u_{\eps,1}^{\alpha}+\alpha t^{\alpha-1}u_{\eps,1}^{\alpha-1}w_1)(t^{\beta}u_{\eps,2}^{\beta}+\beta t^{\beta-1}u_{\eps,2}^{\beta-1}w_2)\nonumber\\
&\qquad\qquad+ c_2\|w_1\|^\alpha\|w_2\|^{\beta} \nonumber \\
&\qquad\ge \io(tu_{\eps,1})^{\alpha}(tu_{\eps,2})^{\beta} + t^{2^*-1}\left(c_3\io\bar{u}_{\eps}^{2^*-1}w_1 + c_4\io\bar{u}_{\eps}^{2^*-1}w_2\right) \nonumber\\
&\qquad\qquad + c_5t^{2^*-2}\io\bar{u}_{\eps}^{2^*-2}w_1w_2+ c_2\|w_1\|^\alpha\|w_2\|^{\beta}. \nonumber
\end{align}
Here and hereafter $c_j$ denotes a positive constant. From inequalities \eqref{bilinear}, \eqref{convex} and Lemma \ref{lem:bn} we get
\begin{align*}
\cJ_\lambda(tu_{\eps}+w)&\leq c_6\Big(t^2+t(\|w_1\|+\|w_2\|)+\|w_1\|^2+\|w_2\|^2 \\
&\quad +\eps^\frac{N-2}{2} t^{2^*-1}(\|w_1\|+\|w_2\|)\Big) - c_7(t^{2^*}+\|w_1\|^{2^*}+\|w_2\|^{2^*}).
\end{align*}
So, as $t^{2^*-1}\|w_i\|\le t^{2^*}+\|w_i\|^{2^*}$, there exists $R>0$ such that, for every $\eps$ small enough,
\begin{equation} \label{0}
\cJ_\lambda(tu_\eps+w) \le 0\qquad\text{if\; }t\ge R\text{\; or\; }\|w\|\ge R.
\end{equation}
From now on, we assume that $t\le R$ and $\|w\|\le R$. We distinguish two cases.

Let $N\geq 5$. Since $B_i(w_i,w_i)\le 0$, using Lemma \ref{lem:bn}, the inequalities \eqref{bilinear}-\eqref{mixed} and the inequalities
\begin{align} 
\max_{s>0}(rs-s^q) &\le C_q\, r^{\frac{q}{q-1}}\quad\text{if } q>1, \label{q}\\
\max_{0\le s_1,s_2\le R}(rs_1s_2-s_1^\alpha s_2^\beta) &\le C_{R,\alpha,\beta}\,\max\{r^{\frac{\alpha}{\alpha-1}},r^{\frac{\beta}{\beta-1}}\} \quad\text{if } \alpha,\beta>1, \label{ab}
\end{align}
which hold true for every $r\ge 0$ ($C_q$ and $C_{R,\alpha,\beta}$ are positive constants that depend only on their subindices), we obtain
\begin{align}
\cJ_\lambda(tu_\eps+w) &\le \cJ_\lambda(tu_\eps) \label{eq:2} \\
&\qquad + O(\eps^\frac{N-2}{2})(\|w_1\|+\|w_2\|)-c_8(\|w_1\|^{2^*}+\|w_2\|^{2^*}) \nonumber \\
&\qquad + O(\eps^2)\|w_1\|\,\|w_2\| - c_9\|w_1\|^\alpha\|w_2\|^\beta \nonumber \\
&\le\cJ_\lambda(tu_\eps) + O(\eps^\frac{N(N-2)}{N+2})+ o(\eps^2).\nonumber 
\end{align}
Note that $\frac{N(N-2)}{N+2}>2$ if $N\geq 5$. So \eqref{eq:1} and \eqref{eq:2} give
\begin{equation*}
\cJ_\lambda(tu_\eps+w)\le\frac1N\left(S_{\infty,\lambda} - C\eps^2+o(\eps^2)\right)^{N/2} + o(\eps^2)
\end{equation*}
for all $t\in(0,R],\,\|w\|\le R$. This inequality, together with \eqref{0}, yields \eqref{eq:claim} for $\eps$ small enough.

Let now $N=4$. Since $\kappa_1,\kappa_2$ are not eigenvalues of $-\Delta$ in $D^{1,2}_0(\Omega)$, there exists $c_{10}>0$ such that $-B_i(w_i,w_i)\ge c_{10}\|w_i\|^2$ for all $w_i\in\wt X_i$, $i=1,2$. Hence, from Lemma \ref{lem:bn} and the inequalities \eqref{bilinear}-\eqref{mixed}, \eqref{q} and \eqref{ab}, we get
\begin{align}
\cJ_\lambda(tu_\eps+w) &\le \cJ_\lambda(tu_\eps) + O(\eps)(\|w_1\|+\|w_2\|)-c_8(\|w_1\|^2+\|w_2\|^2) \nonumber \\
&\qquad + O(\eps^2|\ln \eps|)\|w_1\|\,\|w_2\| -  c_9\|w_1\|^\alpha\|w_2\|^\beta \nonumber \\
&\le \cJ_\lambda(tu_\eps) + O(\eps^2).\nonumber 
\end{align}
Combining this inequality with \eqref{eq:1} gives
\begin{equation*}
\cJ_\lambda(tu_\eps+w)\le\frac14\left(S_{\infty,\lambda} - C\eps^2|\ln\eps|+ O(\eps^2)\right)^2 + O(\eps^2),
\end{equation*}
for all $t\in(0,R],\,\|w\|\le R$. This inequality, together with \eqref{0}, yields \eqref{eq:claim} for $\eps$ small enough.

Using Proposition \ref{bbf} and Lemma \ref{below} we may now proceed as in the proof of Theorem \ref{mainthm} to first obtain a nontrivial solution at some level $c<\frac1NS_{\infty,\lambda}^{N/2}$ and then show that there is a ground state $\bar u$ satisfying $0<B(\bar u,\bar u)<\min\{B_1(\bar w_1,\bar w_1),B_2(\bar w_2,\bar w_2)\}$.
\end{proof}

\section{Synchronized solutions} \label{sec:synchronized}

If $\kappa_1=\kappa_2$ one may look for \emph{synchronized solutions} to the system \eqref{syst}, i.e., solutions of the form $(sw,tw)$ with $s,t\in\r$. In this case, solutions exist for any $\lambda>0$.

\begin{lemma} \label{lem:synchronized}
Let $\kappa_1=\kappa_2$ and let $w$ be a nontrivial solution to equation \eqref{eq}. Then, there exist $s,t>0$ such that $(sw,tw)$ is a solution to the system \eqref{syst} if and only if there exists $r>0$ such that
\begin{equation} \label{eq:h}
h(r) := \mu_1 r^{p-2} + \lambda\alpha r^{\alpha -2}-\lambda\beta r^\alpha -\mu_2=0.
\end{equation}
\end{lemma}

\begin{proof}
The proof of \eqref{eq:h} is the same as that of \cite[Lemma 4.1]{cf} with $2^*$ replaced by $p$. 
\end{proof}

\begin{theorem}
Let $\kappa_1=\kappa_2$, and assume that \eqref{eq:h} holds true for some $r>0$.
\begin{itemize}
\item[$(a)$] If $p\in(2,2^*)$, then the system \eqref{syst} has infinitely many fully nontrivial synchronized solutions.
\item[$(b)$] Let $N\geq 4$ and assume that $\kappa_1$ is not an eigenvalue of $-\Delta$ in $D^{1,2}_0(\Omega)$ if $N=4$. Then, if $p=2^*$, the system \eqref{syst} has a fully nontrivial synchronized solution.
\end{itemize}
\end{theorem}

\begin{proof}
$(a)$ It is well known that the equation \eqref{eq} has infinitely many nontrivial solutions if $p\in(2,2^*)$; see, e.g., \cite[Theorem 3.2]{sw}

$(b)$ It is shown in \cite[Theorem 3.6]{sww} that, under the given assumptions, the equation \eqref{eq} has a ground state solution for $p=2^*$.
\end{proof}

\begin{remark} \emph{
Since 
\[
h(r) = r^{\alpha-2}(\mu_1 r^{\beta} + \lambda(\alpha -\beta r^2) -\mu_2r^{2-\alpha})
\]
and 
\[
h(r) = r^{\alpha}(\mu_1 r^{\beta-2} + \lambda(\alpha r^{-2} -\beta) -\mu_2r^{-\alpha}), 
\]
we see that $h(r)>0$ for small $r>0$ if either $\alpha<2$ or $\alpha=2$ and $\lambda > \mu_2/2$, and that $h(r)<0$ for large $r$ if either $\beta<2$ or $\beta=2$ and $\lambda > \mu_1/2$. So in all of these cases we have that $h(r)=0$ for some $r>0$.}

\emph{Note, in particular, that if $N\ge 6$, then necessarily $\alpha,\beta<2$.
}
\end{remark}

\bigskip

\begin{flushleft}
\textbf{Mónica Clapp}\\
Instituto de Matemáticas\\
Universidad Nacional Autónoma de México\\
Circuito Exterior, Ciudad Universitaria\\
04510 Coyoacán, Ciudad de México, Mexico\\
\texttt{monica.clapp@im.unam.mx} 
\medskip

\textbf{Andrzej Szulkin}\\
Department of Mathematics\\
Stockholm University\\
106 91 Stockholm, Sweden\\
\texttt{andrzejs@math.su.se} 
\end{flushleft}

\end{document}